\documentclass[12pt]{article}

\usepackage[english]{babel}

\usepackage[letterpaper,top=2cm,bottom=2cm,left=3cm,right=3cm,marginparwidth=1.75cm]{geometry}

\usepackage{amsmath}
\usepackage{amssymb}
\usepackage{amsthm}
\usepackage{graphicx}
\usepackage[colorlinks=true, allcolors=blue]{hyperref}

\theoremstyle{plain}
\newtheorem{theorem}{Theorem}
\newtheorem{lemma}[theorem]{Lemma}
\newtheorem{proposition}[theorem]{Proposition}

\theoremstyle{definition}

\newtheorem{example}[theorem]{Example}

\title{Optimal codes in the Stiefel manifold}
\author{John Jasper\thanks{Department of Mathematics and Statistics, Air Force Institute of Technology, Wright-Patterson AFB, Ohio, USA} \and Nathan Mankovich\thanks{Image Processing Laboratory, University of Valencia, Valencia, Spain} \and Dustin G.\ Mixon\thanks{Department of Mathematics, The Ohio State University, Columbus, Ohio, USA} \thanks{Translational Data Analytics Institute, The Ohio State University, Columbus, Ohio, USA}}
\date{}

\begin{document}
\maketitle

\begin{abstract}
We consider the coding problem in the Stiefel manifold with chordal distance.
After considering various low-dimensional instances of this problem, we use Rankin's bounds on spherical codes to prove upper bounds on the minimum distance of a Stiefel code, and then we construct several examples of codes that achieve equality in these bounds.
\end{abstract}

\section{Introduction}

Given $F\in\{\mathbb{R},\mathbb{C}\}$ and $d,r\in\mathbb{N}$ with $d\geq r$, consider the \textbf{Stiefel manifold}
\[
\operatorname{St}_F(d,r)
:=\{X\in F^{d\times r}:X^*X=I_r\}.
\]
We view this manifold as a metric space with the \textbf{chordal distance} induced by the Frobenius norm:
\[
(X,Y)
\longmapsto\|X-Y\|_{\operatorname{Fro}}.
\]
Given $n\in\mathbb{N}$ with $n\geq2$, we wish to solve the following coding problem:
\[
\text{maximize}
\qquad
\min_{\substack{i,j\in[n]\\i\neq j}}
\|X_i-X_j\|_{\operatorname{Fro}}
\qquad
\text{subject to}
\qquad
X_1,\ldots,X_n\in\operatorname{St}_F(d,r).
\]

This coding problem naturally emerges in the context of wireless communication~\cite{HussienEtal:15}.
In this setting, each $X_i$ represents a codeword that is transmitted in concert by $r$ different antennas over the course of $d$ consecutive time slots.
The Stiefel constraint models an antenna array in which the antennas transmit orthogonally with the same energy.
By maximizing the minimum chordal distance, we enable a reciever array to distinguish the codewords, even after corruption by noise.
In~\cite{HussienEtal:15}, Stiefel codes were constructed by numerical optimization, but there is no guarantee that the resulting codes are globally optimal.
In this paper, we construct several examples of optimal Stiefel codes by adapting ideas from the better-studied setting of Grassmannian codes~\cite{ConwayHS:96,FickusM:15}.

In the next section, we examine various low-dimensional instances of our problem.
In Sections~3 and~4, we introduce the Stiefel simplex and orthoplex bounds, respectively, and we construct several examples of codes that achieve equality in these bounds.
We conclude in Section~5 with a discussion.

\section{Optimal codes in small dimensions}

Our coding problem is relatively simple in certain low-dimensional settings.

\begin{example}
\label{ex.R_1_1}
When $F=\mathbb{R}$ and $d=r=1$, the Stiefel manifold is a two-point space $\{\pm1\}$.
The optimal codes with $n=2$ take $X_1=\pm1$ and $X_2=\mp1$, and for all other choices of $n$, every code is optimal.
\end{example}

\begin{example}
\label{ex.circle}
When $F=\mathbb{R}$, $d=2$, and $r=1$, the Stiefel manifold is the unit circle in $\mathbb{R}^2$.
Given $n\geq 2$ points in the circle, consider the angles $\{\theta_k\}_{k=1}^n$ between adjacent points.
Then the minimum chordal distance is achieved by the points that share the smallest angle $\theta_k$, and furthermore,
\[
\min_k\theta_k
\leq\frac{1}{n}\sum_{k=1}^n\theta_k
=\frac{2\pi}{n},
\]
with equality precisely when $\theta_k=\frac{2\pi}{n}$ for all $k$.
It follows that the optimal codes in the circle consist of uniformly spaced points (as one might expect).
\end{example}

\begin{example}
When $F=\mathbb{C}$ and $d=r=1$, the Stiefel manifold is $U(1)$, i.e., the unit circle in $\mathbb{C}$.
By Example~\ref{ex.circle}, it follows that the optimal codes consist of uniformly spaced points, e.g., the cyclic subgroup $C_n\leq U(1)$.
\end{example}

\begin{example}
When $F=\mathbb{R}$ and $d=r=2$, the Stiefel manifold is $O(2)$, which is the union of two disjoint circles of radius $\sqrt{2}$ centered at the zero matrix:
\[
SO(2)
=\left\{\left[\begin{smallmatrix}a&-b\\b&\phantom{-}a\end{smallmatrix}\right]:a^2+b^2=1\right\}
\cup\left\{\left[\begin{smallmatrix}a&\phantom{-}b\\b&-a\end{smallmatrix}\right]:a^2+b^2=1\right\}
=:C_+\cup C_-.
\]
Note that $C_+$ and $C_-$ reside in orthogonally complementary subspaces of $\mathbb{R}^{2\times 2}$, and so every point in $C_+$ has distance $2$ from every point in $C_-$.
Given $n\geq2$, there must be some $a,b\in\mathbb{N}\cup\{0\}$ with $a+b=n$ for which the optimal code consists of $a$ points in $C_+$ and $b$ points in $C_-$.
By Example~\ref{ex.circle}, we may assume the points in each circle are uniformly distributed, and so the minimum distance in the code is $\sqrt{4-4\cos(\frac{2\pi}{k(a,b)})}$, where
\[
k(a,b)=
\left\{
\begin{array}{cl}
\max\{a,b\} & \text{if }\min\{a,b\}=0\\
\max\{a,b,4\} & \text{if }\min\{a,b\}>0.
\end{array}
\right.
\]
Let $k_n$ denote the minimum $k(a,b)$ subject to $a+b=n$ and $a,b\in\mathbb{N}\cup\{0\}$.
Then the largest possible minimum distance between $n$ points in $O(2)$ is $\sqrt{4-4\cos(\frac{2\pi}{k_n})}$, where
\[
\begin{tabular}{c|ccccccccccc}
$n$ & 2 & 3 & 4 & 5 & 6 & 7 & 8 & 9 & 10 & 11 & 12 \\\hline
$k_n$ & 2 & 3 & 4 & 4 & 4 & 4 & 4 & 5 & 5 & 6 & 6
\end{tabular}
\]
and the pattern stabilizes with $k_n=\lceil \frac{n}{2}\rceil$ for $n\geq8$.
For various choices of $n$, some optimizers are given by the cyclic subgroups $C_2,C_3,C_4\leq SO(2)$, as well as the dihedral subgroups $D_2,D_3,D_4,\ldots\leq O(2)$.
\end{example}

\begin{example}
When $F=\mathbb{C}$ and $d=r=2$, the Stiefel manifold is $U(2)$.
In this case, \cite{LiangX:02} reports optimal codes for each $n\leq 16$.
In terms of the current paper, the optimal codes with $n\leq 9$ achieve equality in Theorem~\ref{thm.stiefel simplex bound}, whereas the optimal codes with $n\in(9,16]$ achieve equality in Theorem~\ref{thm.stiefel orthoplex bound} (cf.\ Proposition~2 in~\cite{LiangX:02}).
\end{example}

\begin{example}
When $F=\mathbb{R}$, $d=3$, and $r=1$, the Stiefel manifold is the unit sphere in $\mathbb{R}^3$, and so our coding problem reduces to the famous \textit{Tammes problem}~\cite{Tammes:30}, for which exact solutions are only known for $n\leq 14$ and $n=24$ (see~\cite{CohnLL:24} and references therein), although many more putatively optimal codes have been discovered by numerical optimization~\cite{Sloane:online,Cohn:online}.
\end{example}

\begin{example}
\label{ex.sphere}
When $r=1$, then the Stiefel manifold is the unit sphere in $F^d$, which is isomorphic as a metric space to $S^{md-1}$ with chordal distance, where $m:=[F:\mathbb{R}]$ is the degree of the field extension $F/\mathbb{R}$.
In this setting, codes are known as \textit{spherical codes}~\cite{EricsonZ:01}.
There are four known infinite families of optimal codes in $S^{d-1}$ with $d\geq 2$:
\begin{itemize}
\item[(i)]
$d=2$, $n\geq 2$: uniformly distributed points on the circle; see Example~\ref{ex.circle}.
\item[(ii)]
$d\geq2$, $n\in[2,d+1]$: vertices of a simplex centered at the origin; see Proposition~\ref{prop.rankin simplex}.
\item[(iii)]
$d\geq2$, $n\in(d+1,2d]$: equality in the \textit{orthoplex bound}, e.g., by taking any $n$ of the $2d$ vertices of an orthoplex centered at the origin; see Proposition~\ref{prop.rankin orthoplex}.
\item[(iv)]
$d=q(q^3+1)/(q+1)$, $n=(q+1)(q^3+1)$, $q$ a prime power: spherical embedding of the point graph of a generalized quadrangle $(q,q^2)$; see~\cite{CameronGS:78,Levenshtein:92}.
\end{itemize}
In addition to these infinite families, there are finitely many sporadic examples of known optimal spherical codes; see Tables~1.1 and~1.2 in~\cite{CohnLL:24} for an up-to-date list.
\end{example}

The codes in Example~\ref{ex.sphere}(ii) and~(iii) are optimal as a consequence of achieving equality in Rankin's simplex and orthoplex bounds~\cite{Rankin:55}, which we present below.

\begin{proposition}[Rankin's simplex bound]
\label{prop.rankin simplex}
For every $x_1,\ldots,x_n\in S^{d-1}$, it holds that
\[
\min_{\substack{i,j\in[n]\\i\neq j}}\|x_i-x_j\|
\leq\sqrt{\frac{2n}{n-1}},
\]
with equality precisely when $\{x_i\}_{i\in[n]}$ form the vertices of a regular simplex centered at the origin.
Furthermore, such a code exists only if $n\leq d+1$.
\end{proposition}

\begin{proposition}[Rankin's orthoplex bound]
\label{prop.rankin orthoplex}
For every $x_1,\ldots,x_n\in S^{d-1}$ with $n>d+1$, it holds that
\[
\min_{\substack{i,j\in[n]\\i\neq j}}\|x_i-x_j\|
\leq\sqrt{2}.
\]
Equality is achievable only if $n\leq 2d$, in which case equality is achieved, for example, by any $n$ of the $2d$ vertices of an orthoplex centered at the origin, i.e., an orthonormal basis union its negative.
\end{proposition}

In the following sections, we apply these bounds to obtain tight bounds for our coding problem over more general Stiefel manifolds.

\section{Stiefel simplex codes}

Here and throughout, $m:=[F:\mathbb{R}]$ is the degree of the field extension $F/\mathbb{R}$.

\begin{theorem}[Stiefel simplex bound]
\label{thm.stiefel simplex bound}
For every $X_1,\ldots,X_n\in\operatorname{St}_F(d,r)$, it holds that
\[
\min_{\substack{i,j\in[n]\\i\neq j}}
\|X_i-X_j\|_{\operatorname{Fro}}
\leq\sqrt{\frac{2rn}{n-1}},
\]
with equality precisely when
\[
\|X_i-X_j\|_{\operatorname{Fro}}
=\|X_k-X_l\|_{\operatorname{Fro}}
\quad
\forall \, i,j,k,l\in[n],\, i\neq j,\, k\neq l
\quad\quad
\text{and}
\quad\quad
\sum_{i=1}^n X_i=0.
\]
Furthermore, such a code exists only if $n\leq mdr+1$.
\end{theorem}

\begin{proof}
$\operatorname{St}_F(d,r)$ is contained in a sphere in $F^{d\times r}$ of Frobenius radius $\sqrt{r}$ centered at the zero matrix.
By relaxing our coding problem to this sphere, we inherit Rankin's simplex bound, from which our result follows.
\end{proof}

We refer to codes that achieve equality in the Stiefel simplex bound as \textbf{real (resp.\ complex) $(d,r,n)$-Stiefel simplex codes (SSCs)}.
In this section, we provide several constructions of such codes.
First, we describe various ``new from old'' constructions.

\begin{lemma}\
\begin{itemize}
\item[(a)]
If there exists a real $(d,r,n)$-Stiefel simplex code, then there exist a complex $(d,r,n)$-Stiefel simplex code.
\item[(b)]
If there exists a real (resp.\ complex) $(d,r,n)$-Stiefel simplex code, then there exists a real (resp.\ complex) $(d+1,r,n)$-Stiefel simplex code.
\item[(c)]
If there exists a real (resp.\ complex) $(d,r,n)$-Stiefel simplex code and $k\in\mathbb{N}$, then there exists a real (resp.\ complex) $(kd,kr,n)$-Stiefel simplex code.
\item[(d)]
If there exists a complex $(d,r,n)$-Stiefel simplex code and $k\in\mathbb{N}$, then there exists a real $(2d,2r,n)$-Stiefel simplex code.
\end{itemize}
\end{lemma}

\begin{proof}
For~(a), given an SSC $\{X_i\}_{i\in [n]}$ in $\operatorname{St}_\mathbb{R}(d,r)$, we observe that these matrices also satisfy the conditions to be an SSC in $\operatorname{St}_\mathbb{C}(d,r)$.
For~(b), given an SSC $\{X_i\}_{i\in [n]}$ in $\operatorname{St}_F(d,r)$, append a $(d+1)$st row of zeros to each $X_i$ to form $Y_i$.
Then $\{Y_i\}_{i\in[n]}$ is an SSC in $\operatorname{St}_F(d+1,r)$.
For~(c), given an SSC $\{X_i\}_{i\in [n]}$ in $\operatorname{St}_F(d,r)$, take $Y_i:=I_k\otimes X_i$.
Then $\{Y_i\}_{i\in[n]}$ is an SSC in $\operatorname{St}_F(kd,kr)$. 
For~(d), given an SSC $\{X_i\}_{i\in [n]}$ in $\operatorname{St}_{\mathbb{C}}(d,r)$, replace each (complex) entry $(X_i)_{ab}$ of $X_i$ with the corresponding real $2\times 2$ matrix
\[
(Y_i)_{ab}
:=\left[\begin{array}{rr}
\operatorname{Re}(X_i)_{ab}&-\operatorname{Im}(X_i)_{ab}\\
\operatorname{Im}(X_i)_{ab}&\operatorname{Re}(X_i)_{ab}
\end{array}\right]
\]
Then $\{Y_i\}_{i\in[n]}$ is an SSC in $\operatorname{St}_{\mathbb{R}}(2d,2r)$. 
\end{proof}

Next, we study the extreme values of $r$.
We start with the simpler case of $r=1$.

\begin{lemma}
\label{lem.r=1}
There exists a $(d,1,n)$-Stiefel simplex code if and only if $n\leq md+1$.
\end{lemma}

\begin{proof}
Since $\operatorname{St}_F(d,1)$ is isomorphic as a metric space to the unit sphere in $\mathbb{R}^{md}$ with chordal distance, the result follows from the classical fact that there exists a regular simplex on $n$ vertices in $\mathbb{R}^{md}$ if and only if $n\leq md+1$.
\end{proof}

Next, we consider the other extreme where $r=d$.
Recall that the \textbf{Radon--Hurwitz number} $\rho_F(d)$ is the maximum dimension of a real subspace of $F^{d\times d}$ consisting of scalar multiples of orthogonal (resp.\ unitary) matrices~\cite{Radon:22,Hurwitz:23,Adams:62,AdamsLP:65}, and is given by
\[
\rho_F(d)
=\left\{
\begin{array}{ll}
8b+2^c&\text{if }F=\mathbb{R}\\
8b+2c+2&\text{if }F=\mathbb{C},
\end{array}
\right.
\]
where $d=(2a+1)\cdot 2^{4b+c}$ for some nonnegative integers $a$, $b$, and $c$ with $c\leq 3$.
Such matrix spaces were used recently in~\cite{FickusGI:24} to characterize a certain class of codes in the Grassmannian.

\begin{lemma}
If $n\leq \rho_F(d)+1$, then there exists a $(d,d,n)$-Stiefel simplex code.
\end{lemma}

\begin{proof}
There exists a $\rho_F(d)$-dimensional real subspace of $F^{d\times d}$ in which every matrix of Frobenius norm $\sqrt{d}$ is orthogonal (resp.\ unitary).
It suffices to select any $n$ points in this sphere that form the vertices of a centered regular simplex.
\end{proof}

While the previous result is limited by the fact that $\rho_F(d)=O(\log d)$, the following result gives that there also exists a $(d,d,n)$-SSC with $n=d+1$.

\begin{lemma}
There exists a real $(d,d,d+1)$-Stiefel simplex code.
\end{lemma}

\begin{proof}
Take any group $G$ of order $d+1$, and let $\rho\colon G\to O(d+1)$ denote the left regular representation of $G$.
The image of this representation consists of $d+1$ orthogonal matrices that are pairwise orthogonal.
Furthermore, the trivial representation is a constituent representation with multiplicity $1$, and so there exists $Q\in O(d+1)$ and a representation $\pi\colon G\to O(d)$ such that 
\[
Q\rho(g)Q^{-1}
=\left[\begin{array}{cc}
1&0_d^\top\\
0_d&\pi(g)
\end{array}\right]
\]
for all $g\in G$.
Next, for every $g,g'\in G$, we use the Kronecker delta to write
\[
(d+1)\cdot\delta_{g,g'}
=\operatorname{Tr}\big(\rho(g)^\top\rho(g')\big)
=\operatorname{Tr}\big((Q\rho(g)Q^\top)^\top(Q\rho(g')Q^\top)\big)
=1+\operatorname{Tr}\big(\pi(g)^\top\pi(g')\big).
\]
Thus, if $g\neq g'$, then
\[
\|\pi(g)-\pi(g')\|_{\operatorname{Fro}}^2
=\|\pi(g)\|_{\operatorname{Fro}}^2-2\operatorname{Tr}\big(\pi(g)^\top\pi(g')\big)+\|\pi(g')\|_{\operatorname{Fro}}^2
=2d+2.
\]
Furthermore, since the trivial representation is not consituent in $\pi$, it follows that $\sum_{g\in G}\pi(g)=0$.
Overall, $\{\pi(g)\}_{g\in G}$ is a $(d,d,d+1)$-SSC.
\end{proof}

Next, we construct SSCs by \textit{symplectic lifting}.

\begin{lemma}
If $d$ is even and $n\leq md+1$, then there exists a $(d,2,n)$-Stiefel simplex code.
\end{lemma}

\begin{proof}
Since $n\leq md+1$, then by Lemma~\ref{lem.r=1}, there exists a $(d,1,n)$-SSC consisting of unit vectors $x_1,\ldots,x_n\in F^d$.
Since $d$ is even, we may consider the block matrix
\[
A:=\left[\begin{array}{cc}
0&-I_{d/2}\\
I_{d/2}&0
\end{array}\right].
\]
Then for each $i\in[n]$, the matrix $X_i:=[x_i,\overline{Ax_i}]\in F^{d\times 2}$ has orthonormal columns, and therefore resides in $\operatorname{St}_F(d,r)$. Furthermore, for each $i,j\in[n]$ with $i\neq j$, we have
\[
\|X_i-X_j\|_{\operatorname{Fro}}^2
=\|x_i-x_j\|^2+\|\overline{Ax_i}-\overline{Ax_j}\|^2
=2\|x_i-x_j\|^2
=\frac{4n}{n-1},
\]
where the last step follows from equality in Rankin's simplex bound.
Finally, $\{X_i\}_{i\in[n]}$ is centered since $\{x_i\}_{i\in[n]}$ is centered.
Overall, $\{x_i\}_{i\in[n]}$ is a $(d,2,n)$-SSC.
\end{proof}

Next, we show how to combine an SSC in $\operatorname{St}_F(d,r)$ with a resolvable balanced incomplete block design (BIBD) to obtain a larger SSC.
In order to disambiguate from the well-established BIBD parameter $r$, we start with an SSC in $\operatorname{St}_F(d,s)$.

\begin{theorem}
\label{thm.resolvable bibd}
If there exists a real (resp.\ complex) $(d,s,n)$-Stiefel simplex code and a resolvable $(v,b,r,k,\lambda)$-balanced incomplete block design with $k=n$ and $\lambda=1$, then there exists a real (resp.\ complex) $(bd,rs,v)$-Stiefel simplex code.
\end{theorem}

As an example, we combine the real $(1,1,2)$-SSC $\{\pm1\}$ from Example~\ref{ex.R_1_1} with the resolvable BIBD whose points and blocks are the vertices and edges of the complete graph on $4$ vertices.
The incidence matrix of this BIBD is given by
\[
\left[\begin{array}{cc|cc|cc}
1&0&1&0&1&0\\
1&0&0&1&0&1\\
0&1&1&0&0&1\\
0&1&0&1&1&0
\end{array}\right],
\]
where the rows are indexed by the $v=4$ points and the columns are indexed by the $b=6$ blocks.
Here, we arranged the blocks into $r=3$ pairs that each partition the point set.
This is known as a \textit{resolution} of the resolvable BIBD.
Each column has a total of $k=2$ ones, and each pair of distinct rows has $\lambda=1$ one in common.
Since $k=n$, we can bijectively assign the members of our SSC to the ones in each column, for example:
\[
\left[\begin{array}{cc|cc|cc}
+&0&-&0&+&0\\
-&0&0&+&0&-\\
0&+&+&0&0&+\\
0&-&0&-&-&0
\end{array}\right].
\]
We now convert each row of this SSC-weighted incidence matrix into a point in the $6\times 3$ Stiefel manifold:
\[
\left[\begin{array}{ccc}
+&&\\
0&&\\
&-&\\
&0&\\
&&+\\
&&0
\end{array}\right],
\qquad
\left[\begin{array}{ccc}
-&&\\
0&&\\
&0&\\
&+&\\
&&0\\
&&-
\end{array}\right],
\qquad
\left[\begin{array}{ccc}
0&&\\
+&&\\
&+&\\
&0&\\
&&0\\
&&+
\end{array}\right],
\qquad
\left[\begin{array}{ccc}
0&&\\
-&&\\
&0&\\
&-&\\
&&-\\
&&0
\end{array}\right].
\]
(Here, blank entries are zeros.)
One may verify that this is a real $(6,3,4)$-SSC.

\begin{proof}[Proof of Theorem~\ref{thm.resolvable bibd}]
Let $\{X_i\}_{i\in[n]}$ denote the given $(d,s,n)$-SSC, and let $\mathcal{V}$ and $\mathcal{B}$ denote the point and block sets of the given BIBD.
For each block $B\in\mathcal{B}$, select a bijection
\[
X_B\colon B\to\{X_i:i\in[n]\}.
\]
(Such a bijection exists since $k=n$ by assumption.)
Next, since our BIBD is resolvable, there exists a \textit{resolution} $\mathcal{R}$, that is, a partition of $\mathcal{B}$ with the property that every $\mathcal{C}\in\mathcal{R}$ is a partition of $\mathcal{V}$.
Then for each $p\in\mathcal{V}$, we define $Y_p\in(F^{d\times s})^{\mathcal{B}\times\mathcal{R}}$ by 
\[
(Y_p)_{B,\mathcal{C}}
:=1_{\{p\in B\}}\cdot 1_{\{B\in\mathcal{C}\}}\cdot X_B(p).
\]
(Here, $1_{\{P\}}$ is defined to be $1$ whenever $P$ holds, and is otherwise $0$.)
One may argue in cases that the $rs$ columns of $Y_p$ are orthonormal: for columns with the same $\mathcal{C}$, the inner product reduces to an inner product between (orthonormal) columns of $X_B(p)$, whereas columns with different $\mathcal{C}$'s have disjoint support.
Thus, the matrices $\{Y_p\}_{p\in\mathcal{V}}$ reside in $\operatorname{St}_F(bd,rs)$, and it remains to verify that they achieve equality in the Stiefel simplex bound.
To this end, fix any $p,q\in\mathcal{V}$ with $p\neq q$.
Then
\begin{align*}
\operatorname{Re}\operatorname{Tr}(Y_p^*Y_q)
&=\operatorname{Re}\operatorname{Tr}\bigg(\sum_{B\in\mathcal{B}}\sum_{\mathcal{C}\in\mathcal{R}}(Y_p)_{B,\mathcal{C}}^*(Y_q)_{B,\mathcal{C}}\bigg)\\
&=\operatorname{Re}\operatorname{Tr}\bigg(\sum_{B\in\mathcal{B}}\sum_{\mathcal{C}\in\mathcal{R}}1_{\{p\in B\}}1_{\{q\in B\}}1_{\{B\in\mathcal{C}\}} X_B(p)^*X_B(q)\bigg)\\
&=\sum_{B\in\mathcal{B}}1_{\{p\in B\}}1_{\{q\in B\}}
\sum_{\mathcal{C}\in\mathcal{R}}
1_{\{B\in\mathcal{C}\}}\operatorname{Re}\operatorname{Tr}\big(X_B(p)^*X_B(q)\big).
\end{align*}
Since $X_B$ is a bijection and $\{X_i\}_{i\in[n]}$ is a $(d,s,n)$-SSC by assumption, we have 
\[
\operatorname{Re}\operatorname{Tr}\big(X_B(p)^*X_B(q)\big)
=-\frac{s}{n-1}
=-\frac{rs}{v-1},
\]
where the last step follows from our assumption that $k=n$ and the standard result about BIBDs that $r=\frac{v-1}{k-1}$.
Thus,
\begin{align*}
\operatorname{Re}\operatorname{Tr}(Y_p^*Y_q)
&=\sum_{B\in\mathcal{B}}1_{\{p\in B\}}1_{\{q\in B\}}
\sum_{\mathcal{C}\in\mathcal{R}}
1_{\{B\in\mathcal{C}\}}\operatorname{Re}\operatorname{Tr}\big(X_B(p)^*X_B(q)\big)\\
&=-\frac{rs}{v-1}\sum_{B\in\mathcal{B}}1_{\{p\in B\}}1_{\{q\in B\}}
\sum_{\mathcal{C}\in\mathcal{R}}
1_{\{B\in\mathcal{C}\}}\\
&=-\frac{rs}{v-1},
\end{align*}
where the last step follows from the facts that every block $B\in\mathcal{B}$ is contained in a unique parallel class $\mathcal{C}\in\mathcal{R}$, and every pair of distinct points $p,q\in\mathcal{V}$ determines a unique block $B\in\mathcal{B}$.
Finally, we have
\[
\|Y_p-Y_q\|_{\operatorname{Fro}}^2
=\|Y_p\|_{\operatorname{Fro}}^2-2\operatorname{Re}\operatorname{Tr}(Y_p^*Y_q)+\|Y_q\|_{\operatorname{Fro}}^2
=2rs+2\cdot\frac{rs}{v-1}
=\frac{2rsv}{v-1}.
\]
Since $p$ and $q$ were arbitrary, it follows that $\{Y_p\}_{p\in\mathcal{V}}$ achieves equality in the Stiefel simplex bound.
\end{proof}

\section{Stiefel orthoplex codes}

By modifying our proof of Theorem~\ref{thm.stiefel orthoplex bound} to instead pass to Rankin's orthoplex bound, we obtain the following result.

\begin{theorem}[Stiefel orthoplex bound]
\label{thm.stiefel orthoplex bound}
For every $X_1,\ldots,X_n\in \operatorname{St}_F(d,r)$ with $n>mdr+1$, it holds that
\[
\min_{\substack{i,j\in[n]\\i\neq j}}
\|X_i-X_j\|_{\operatorname{Fro}}
\leq\sqrt{2r}.
\]
Equality is achievable only if $n\leq 2mdr$.
\end{theorem}

We refer to codes that achieve equality in the Stiefel orthoplex bound as \textbf{real (resp.\ complex) $(d,r,n)$-Stiefel orthoplex codes (SOCs)}.
In this section, we construct several of such codes.
We first show that complex SOCs exist for \textit{all} possible choices of parameters.

\begin{theorem}
\label{thm.complex orthoplex}
There exists a complex $(d,r,n)$-Stiefel orthoplex code if and only if $n\in(2dr+1,4dr]$.
\end{theorem}

\begin{proof}
The ``only if'' part follows from the Stiefel orthoplex bound.
For the ``if'' part, it suffices to construct a complex $(d,r,n)$-SOC with $n=4dr$, since one can then remove points to obtain an SOC with $n\in(2dr+1,4dr)$.
To this end, let $e_k$ denote the $k$th standard basis element, let $T\in\mathbb{C}^{d\times d}$ denote circular translation defined by $Te_k=e_{k+1\bmod d}$, and let $M\in\mathbb{C}^{r\times r}$ denote modulation defined by $Me_k=e^{2\pi i k/r}e_k$.
Consider the action of $C_4\times C_d\times C_r$ on $\mathbb{C}^{d\times r}$ defined by $(a,b,c)\cdot X=i^aT^bXM^{-c}$.
We claim that the orbit of $X_0:=[I_r;\,0_{(d-r)\times r}]$ is a complex $(d,r,4dr)$-SOC.
Since this group acts by isometries, it suffices to verify that
\[
\min_{(a,b,c)\neq(0,0,0)}
\|X_0-(a,b,c)\cdot X_0\|_{\operatorname{Fro}}
=\sqrt{2r},
\]
or equivalently, that
\[
\max_{(a,b,c)\neq(0,0,0)}\operatorname{Re}\operatorname{Tr}(X_0^*i^aT^bX_0M^{-c})=0.
\]
In general, one can express the above trace in terms of the Kronecker delta as
\[
\operatorname{Tr}(X_0^*i^aT^bX_0M^{-c})
=i^a\cdot\delta_{b,0}\cdot\delta_{c,0}\cdot r,
\]
from which the result follows.
\end{proof}

The real case is more complicated.
First, we recall the spherical case in which $r=1$.

\begin{lemma}
There exists a real $(d,1,n)$-Stiefel orthoplex code if and only if $n\in(d+1,2d]$.
\end{lemma}

Next, we treat larger values of $r$ by fomulating a real generalization of our construction in Theorem~\ref{thm.complex orthoplex}.

\begin{theorem}
Assuming the Hadamard conjecture, there exists a real $(d,r,n)$-Stiefel orthoplex code for every $r\not\equiv 1\bmod 4$, $d\geq r$, and $n>d+1$ such that
\[
n
\leq\left\{\begin{array}{cl}
2dr&\text{if }r\equiv0\bmod 4\\
d(r+2)&\text{if }r\equiv2\bmod 4\\
d(r+1)&\text{if }r\equiv3\bmod 4.
\end{array}\right.
\]
\end{theorem}

\begin{proof}
We start by describing a general construction, and then we argue that the ingredients of our construction exist under the hypothesis of the lemma.

Take any Hamming code $C\subseteq\{0,1\}^r$ of minimum distance at least $\frac{r}{2}$.
Let $T\in\mathbb{R}^{d\times d}$ denote circular translation defined by $Te_k=e_{k+1\bmod d}$, and for each $c\in C$, define $D_c\in\mathbb{R}^{d\times r}$ by $D_c(i,i)=(-1)^{c_i}$ for each $i\in[r]$ and $D_c(i,j)=0$ whenever $i\neq j$.
Then $D_c\in\operatorname{St}_\mathbb{R}(d,r)$.
We claim that $\{T^a D_c\}_{a\in\mathbb{Z}/d,\,c\in C}$ is a SOC.
In general, one can express the relevant trace in terms of the Kronecker delta and the Hamming distance as
\[
\operatorname{Tr}\big((T^aD_c)^\top(T^{a'}D_{c'})\big)
=\delta_{a,a'} \sum_{k=1}^r (-1)^{c_i} (-1)^{c'_i}
=\delta_{a,a'} \cdot \big(r-2d(c,c')\big),
\]
which is at most $0$ unless $(a,c)=(a',c')$.
Our intermediate claim follows.

It remains to determine how large $C$ can be.
By the Plotkin bound, $|C|$ is at most $2r$, $r+2$, or $r+1$, depending on whether $r\equiv 0$, $2$, or $3\bmod 4$, respectively; see Corollary~4 of Chapter~2 in~\cite{MacWilliamsS:77}.
Furthermore, Levenshtein proved that these upper bounds are always achievable provided enough Hadamard matrices exist; see Theorem~8 of Chapter~2 in~\cite{MacWilliamsS:77}.
\end{proof}

\section{Discussion}

In this paper, we constructed several optimal codes in the Stiefel manifold with chordal distance, but many open problems remain.
For example, we observed that for small values of $d$, many optimal codes in $O(d)$ and $U(d)$ can be realized as subgroups.
For general $d$, which finite subgroups of $O(d)$ and $U(d)$ are optimal codes?
While other Stiefel manifolds do not have a group structure, they are invariant under the action of linear isometries on the left and phased permutations on the right.
We used a subgroup of these symmetries to construct an orthoplex in every complex Stiefel manifold (see the proof of Theorem~\ref{thm.complex orthoplex}).
Are there more examples of optimal codes that can be realized as the orbit of a finite group?
See~\cite{PitavalT:13} for a few examples that take this approach.
Recall that in Example~\ref{ex.sphere}, we listed the four known infinite families of optimal spherical codes, and in this paper, we managed to realize three of these families as optimal Stiefel codes.
Is there a way to embed the generalized quadrangle family into the Stiefel manifold?
From a numerical standpoint, what are good optimization strategies for finding putatively optimal codes in the Stiefel manifold?
Finally, while our work here focused on achieving equality in the Stiefel simplex and orthoplex bounds, there are other bounds (e.g., those derived in~\cite{Henkel:05,HanR:06}) for which equality might also be achievable.

\section*{Acknowledgments}

This project was initiated at an ICERM workshop on ``Recent Progress on Optimal Point Distributions and Related Fields.''
NM thanks Ignacio Santamaria and Diego Cuevas at Universidad de Cantabria for introducing him to this problem.
JJ was supported in part by NSF DMS 2220320.
NM acknowledges the support of Generalitat Valenciana and the Conselleria d'Innovaci\'{o}, Universitats, Ci\`{e}ncia i Societat Digital, through the project ``AI4CS:\ Artificial Intelligence for complex systems:\ Brain, Earth, Climate, Society'' (CIPROM/2021/56).
DGM was supported in part by NSF DMS 2220304 and an Air Force Summer Faculty Fellowship.
The views expressed are those of the authors and do not reflect the official guidance or position of the United States Government, the Department of Defense, the United States Air Force, or the United States Space Force.

\end{document}